\documentclass[a4paper,11pt]{article}
\usepackage[margin=1in,footskip=0.25in]{geometry}
\usepackage{amsfonts}
\usepackage{amsmath}
\usepackage{amsthm}
\usepackage{amssymb}
\usepackage{algorithm}
\usepackage[noend]{algpseudocode}
\usepackage[hyphens]{url}
\usepackage{hyperref}
\usepackage{graphicx}
\usepackage{subfigure}
\usepackage{xspace}
\usepackage{units}
\usepackage{xcolor}
\usepackage[numbers,sort&compress]{natbib}

\newtheorem{theorem}{Theorem}
\newtheorem{corollary}[theorem]{Corollary}

\newtheorem{lemma}[theorem]{Lemma}
\newtheorem{observation}[theorem]{Observation}

\newtheorem{definition}{Definition}
\newtheorem{openproblem}{Open Problem}
\newtheorem{problem}{Problem}

\newcommand{\Sym}[1]{{\rm Sym}(#1)}
\newcommand{\Alt}[1]{{\rm Alt}(#1)}

\newcommand {\leftpush}  {\includegraphics{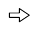}}
\newcommand {\rightpush} {\includegraphics{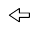}}
\newcommand {\uppush}    {\raisebox{-1pt}{\includegraphics{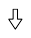}}}
\newcommand {\downpush}  {\raisebox{-1pt}{\includegraphics{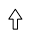}}}

\begin{document}
\title{Pushing Blocks by Sweeping Lines}
\author{
Hugo A.~Akitaya\thanks{University of Massachusetts Lowell, \protect\url{hugo_akitaya@uml.edu}.}\and
Maarten L\"offler\thanks{Department of Information and Computing Sciences, Utrecht University, \protect\url{m.loffler@uu.nl}.}\and
Giovanni Viglietta\thanks{Department of Computer Science and Engineering, University of Aizu, \protect\url{viglietta@gmail.com}.}
}
\date{}

\maketitle

\begin{abstract}
We investigate the reconfiguration of $n$ blocks, or ``tokens'', in the square grid using \emph{line pushes}. A line push is performed from one of the four cardinal directions and pushes all tokens that are maximum in that direction to the opposite direction by one unit. Tokens that are in the way of other tokens are displaced in the same direction, as well.

Similar models of manipulating objects using uniform external forces match the mechanics of existing games and puzzles, such as Mega Maze, 2048 and Labyrinth, and have also been investigated in the context of self-assembly, programmable matter and robotic motion planning.
The problem of obtaining a given shape from a starting configuration is known to be NP-complete.

We show that, for every $n$, there are \emph{sparse} initial configurations of $n$ tokens (i.e., where no two tokens are in the same row or column) that can be rearranged into any $a\times b$ box such that $ab=n$. However, only $1\times k$, $2\times k$ and $3\times 3$ boxes are obtainable from any arbitrary sparse configuration with a matching number of tokens.
We also study the problem of rearranging labeled tokens into a configuration of the same shape, but with permuted tokens.
For every initial ``compact'' configuration of the tokens, we provide a complete characterization of what other configurations can be obtained by means of line pushes.
\end{abstract}

\tableofcontents

\section{Introduction}\label{s:1}
\paragraph{Background.}
Manipulating a set of objects with uniform external forces is a concept that appears  in many game and puzzle mechanics, such as Mega Maze~\cite{MegaMaze}, the 2048 puzzle~\cite{2048-1,2048-2}, Tilt~\cite{BDF+19} and dexterity games such as the Labyrinth marble game~\cite{labyrinth} and Pigs in Clover~\cite{Pigs}. 
It also appears in self-assembly, programmable matter and robotic motion planning, with many applications involving controlling particles in the micro and nanoscale~\cite{becker2013reconfiguring,BDF+19,BGC+20,BLC+19}.
Having the particles being controlled by a uniform external force is of particular interest since it might be unfeasible to control them individually due to their small scale.
The most studied model for self-assembly using uniform external forces is the \emph{tilt} model.
Objects called ``tokens'' are in a 2D board where some locations are marked ``blocked''. 
A tilt move moves all tokens maximally in one of the four cardinal directions, stopping the movement only if there is a collision with a blocked position or with another token~\cite{becker2013reconfiguring}.
Another model that resembles the dexterity games model uses a \emph{single step} to control tokens via uniform signals to move a single unit in one of the cardinal directions~\cite{becker2013massive, UCNC2023, caballero2020building}. 

Akitaya et al.~\cite{trash} introduced the \emph{trash compaction} problem that displaces tokens in a square grid via a \emph{line push}, or simply \emph{push}.
A push is also caused by an external force, but unlike the tilt, each token moves by at most one unit in the direction of the force, perhaps better approximating a dexterity game model.
Informally, a push is applied in one of the four cardinal directions from which we sweep an axis-aligned line. The first tokens hit by the line are displaced by one unit; in turn, these tokens might displace other tokens, and so on. Figure~\ref{fig:example-initial} illustrates an example.
If we consider configurations equivalent under translation, the same push operation can be seen as placing a line barrier on one of the sides of the bounding box and applying a uniform force pushing all tokens towards the barrier by at most one unit.  Note that pushes are not necessarily reversible.

\begin {figure}
  \centering
  \includegraphics [scale=1.5] {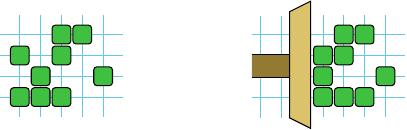}
  \caption {A configuration of tokens in the square lattice (left), and the configuration after a $\leftpush$ push (right). Tokens can be thought of as being pushed by a line coming from the left or, equivalently, as ``falling'' to the left without exiting their bounding box.}
  \label {fig:example-initial}
\end {figure}

The trash compaction problem gives an initial configuration of $1\times 1$ tokens in the square grid and asks whether it can be reconfigured via pushes into a rectangular box of given dimensions.
It is shown in~\cite{trash} that the problem is NP-complete in general, but it is polynomial-time solvable for $2\times k$ rectangles, where $k$ is an arbitrary constant.

\paragraph{Our contributions.}
In this paper, we consider reconfiguration problems using pushes in two scenarios: \emph{labeled} and \emph{unlabeled}.
Two tokens with the same label are indistinguishable;
a configuration is unlabeled if all tokens have the same label, and labeled otherwise.

In Section~\ref{s:2} we make some preliminary observations, where we study certain important configurations called \emph{compact} and the ways they can be reconfigured.

Next, we investigate two types of puzzles.
The first is called \emph{Compaction Puzzle}, and is equivalent to the trash compaction problem.
In Section~\ref{s:3} we show that, for \emph{sparse} configurations of unlabeled tokens (i.e., where no two tokens are in the same row or column), only rectangles of sizes $1\times k$, $2\times k$, and $3\times 3$ can be obtained in general. That is, we give algorithms to push tokens into these rectangles, and we show that there are sparse configurations that cannot be pushed into rectangles of any other size.

The second puzzle is called \emph{Permutation Puzzle}, and the goal is to reconfigure a labeled configuration into another; Figure~\ref{fig:fun2022-puzzle} shows an example. 
In Section~\ref{s:4} we give a complete characterization of which labeled compact configurations can be obtained from one another. Namely, in Section~\ref{s:4.1} we prove that only even permutations of the tokens are possible, and in Section~\ref{s:4.2} we show exactly which even permutations can be obtained, depending on the initial configuration.
Our characterization can be considered a (partial) universality result; that is, after ruling out cases using some easily identifiable necessary conditions, every pair of configurations can be reconfigured into each other (Theorem~\ref{theorem-final} gives a precise statement).

Section~\ref{s:5} concludes the paper with some directions for future research.

A preliminary version of this paper appeared in~\cite{FUN2022}.

\begin {figure}
  \centering
  \includegraphics [scale=1.15] {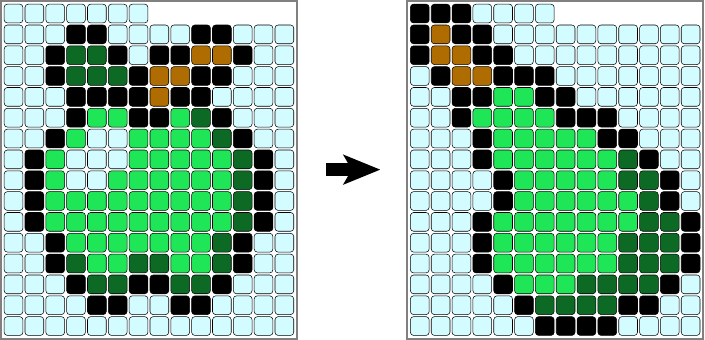}
  \caption {A Permutation Puzzle, where the goal is to rearrange the tokens from the left configuration to the right configuration by means of line pushes. Due to Theorem~\ref{theorem-final}, this puzzle is solvable (the top row has seven empty cells, and so the \emph{core} is empty; see Section~\ref{s:4.2}).}
  \label {fig:fun2022-puzzle}
\end {figure}

\paragraph{Related work.}
There is a rich literature regarding reconfiguration through pushing objects, including sliding-block puzzles \cite{Hearn05, D01-sliding}, block-pushing puzzles~\cite{demaine2003pushing,hearn2005pspace}, the 15-puzzle~\cite{kornhauser1984coordinating}, the 2048 puzzle~\cite{2048-1,2048-2}, etc.
In the block-pushing model, a $1\times 1$ \emph{agent} moves through empty positions in the square lattice and can displace (push) objects that are \emph{free}, while some objects are \emph{blocked} and cannot be moved.
The problem of whether the agent can reach a target position, depending on the model of the push move, is known to be PSPACE-complete.
Note that this operation, albeit being called ``push'', differs from our model since the agent can move specific individual objects.

As previously discussed, models that consider movement of objects by uniform external forces have received some attention due to their applications to programmable matter.
In the tilt model with $1\times 1$ free and blocked objects, deciding whether a given free object can get to a given position is PSPACE-complete~\cite{BGC+20}. This implies the same complexity for the reconfiguration problem.
Other papers considered some form of universality given a certain placement of blocked objects such as reconfiguration~\cite{becker2013reconfiguring,zhang2017rearranging}, particle computations~\cite{BDF+19}, and building shapes when given edges of free objects stick together when in contact~\cite{BLC+19,BGC+20}.
In the single-step model, for a given configuration of blocked objects, reconfiguration is NP-complete with single steps limited to two directions~\cite{caballero2020hardness}, and universality results also exist for a special configuration of blocked objects~\cite{caballero2021fast}.
Note that the hardness of the trash compaction problem~\cite{trash} implies the hardness of the problem of constructing a given shape in the single-step model.

Concerning our line-push model, in addition to the aforementioned paper~\cite{trash} about trash compaction, a second short paper has appeared~\cite{core}, introducing some 2-player games based on the same mechanic.

\section{Definitions and Preliminaries}\label{s:2}
Let $\mathcal{L}$ be the 2D square lattice.
Let $\mathcal{T}$ be a set of $n$ labeled objects called \emph{tokens}, and let $\Sigma$ be the set of labels.
A \emph{configuration} of $\mathcal{T}$ is an arrangement of $\mathcal{T}$ in $\mathcal{L}$ where no two tokens occupy the same position.
Formally, a configuration is a function $C\colon\mathcal{L}\to\Sigma\cup\{{\rm empty}\}$ where the cardinality of the preimage of an element $\ell\in \Sigma$ is the number of tokens labeled $\ell$, and $|C^{-1}(\Sigma)|=|\mathcal T|=n$.
A lattice position $(x,y)$ is \emph{full} if its image is in $\Sigma$, and \emph{empty} otherwise.
Notice that we do not distinguish two tokens that have the same label.
However we may refer to full positions and tokens interchangeably for ease of reference. 
The \emph{bounding box} of $C$ is the minimum rectangular subset of $\mathcal{L}$ containing all full positions. 
In the following, we will identify configurations that are equal under translation.

A \emph{push} is an operation that takes a configuration $C$ and a direction $d\in\{\leftpush, \rightpush, \uppush, \downpush\}$ and returns a configuration $C'$ as follows.
We describe a $\leftpush$ push; the other cases are symmetric.
Informally, a $\leftpush$ push moves all leftmost tokens one unit to the right, further displacing tokens to the right if there are collisions.
Without loss of generality, assume that the lower-left corner of the bounding box is $(0,0)$, applying the appropriate translation otherwise.
For all columns $i\ge 1$ from left to right, and for all rows $j$, if $(i,j)$ is full and $(i-1,j)$ is empty, move the token from $(i,j)$ to $(i-1,j)$.
Finally, translate the configuration making the lower-left corner of the bounding box $(1,0)$.

Note that the last step just translates the configuration, producing an equivalent one.
Then, we can also consider an alternative informal interpretation of a push: 
A $\leftpush$ push places a vertical barrier to the left of $x=0$ and lets all tokens that can move ``fall'' towards the left by one unit. We call this interpretation the ``gravity'' formulation of the puzzle.

Observe that pushes are not necessarily reversible, and every push either does not affect the size of the bounding box, or decreases its area by exactly one row or column.
We denote a sequence of pushes by the respective sequence of directions $\langle d_1 \ldots d_m\rangle$.
We use the notation $d^k$ to express $k$ repetitions of direction $d$.
A configuration is \emph{compressible} if a push can decrease the area of its bounding box, or \emph{incompressible} otherwise.
By definition, it is easy to see that a configuration is incompressible if and only if its bounding box contains a full row and a full column.
If $|\Sigma|=1$, i.e., all tokens have the same label, we call the configuration \emph{unlabeled}. When no restrictions are made on $|\Sigma|$, we say that the configuration is \emph{labeled}.

We can now define our Compaction Puzzle.

\begin{problem}[Compaction Puzzle] \label{prob:compaction}
	Given an unlabeled starting configuration, is there a sequence of pushes that produces a configuration whose tokens form an $a\times b$ box?
\end{problem}

If we start from any configuration and we keep pushing in the two directions $\leftpush$ and $\downpush$ alternately, we eventually reach a configuration for which any push in these two directions produces an equivalent configuration.
We call such a configuration \emph{canonical}.
By definition, a canonical configuration forms an orthogonally convex polyomino with its leftmost column and bottommost row full.
A \emph{compact} configuration is a configuration that can be obtained from a canonical configuration by pushes (see Figure~\ref{fig:canonical}). 
Note that all canonical configurations are also compact.

\begin {figure}
  \centering
  \includegraphics [scale=1.25] {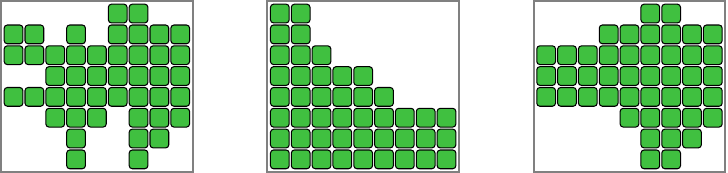}
  \caption {An incompressible non-compact configuration (left); a canonical configuration (center); a compact configuration (right). The central and right configurations are compatible.}
  \label {fig:canonical}
\end {figure}

\begin{observation}
	\label{obs:compact-shape}
	A configuration $C$ is compact if and only if every row $r$ (resp., column) has a single contiguous interval of tokens and its projection on any other row $r'$ (resp., column $c'$) with the same amount or more tokens is contained in the interval of tokens of $r'$ (resp., $c'$).
\end{observation}
\begin{proof}
Let $\mathcal P$ denote the conditions after the ``if and only if''. Observe that all canonical configurations satisfy $\mathcal P$. Also, $\mathcal P$ is easily seen to be preserved by pushes, and therefore all compact configurations satisfy $\mathcal P$, as well.

In order to prove the converse, we first show that if $C$ satisfies $\mathcal P$, then any push performed in $C$ is reversible.
	The remainder of the proof follows from the fact that we can obtain a canonical configuration by applying $\leftpush$ and $\downpush$ pushes.

Without loss of generality, assume a $\leftpush$ push in $C$ and that the lower-left corner of the bounding box is $(0,0)$.
	Let $H$ be the set of rows with a token in the leftmost column $(0,\cdot)$.
	Let $V$ be the set of columns to the right of $(0,\cdot)$ that contain a token in each row in $H$.
	Observe that each token in $C$ must be in $H$ or in $V$, or else there would be a column whose interval of tokens' projection neither contains nor is contained in the interval of tokens of the first column $(0,\cdot)$.
	Note that applying the $\leftpush$ push causes only tokens in $H$ to move.
	We reverse this operation by applying $\langle\, \rightpush^k\,\rangle$, with $k$ chosen so that the right side of the bounding box coincides with the rightmost column in $V$, then applying $\langle\, \leftpush^{k-1}\,\rangle$. 
	By construction, only tokens in $H$ move, and the sequence brings the leftmost token in each row in $H$ back to its original position in column $(0,\cdot)$. This proof works for the labeled as well as the unlabeled case.
\end{proof}

The proof of Observation~\ref{obs:compact-shape} immediately implies the following corollary.

\begin{corollary}
	\label{cor:reversible}
	A push applied to a (labeled) compact configuration is reversible.\qed
\end{corollary}

Note that, by Observation~\ref{obs:compact-shape}, compact configurations are orthogonally convex polyominoes; also, all compact configurations are incompressible.
Moreover, a random sequence of pushes results in a compact configuration with high probability (this certainly happens, for instance, if the sequence of pushes has a subsequence of the form $\langle\, \downpush^{k}\ \leftpush^{k}\, \rangle$ for a large-enough $k$).

The number of tokens in a row or a column of a compact configuration is said to be the \emph{length} of that row or column.
Two compact configurations are \emph{compatible} if they have the same number of rows (resp., columns) of each given length. For example, in Figure~\ref{fig:canonical}, the central and right configurations are compatible.
The observation below shows that the set of compatible compact configurations is closed under pushes. 

\begin{observation}
	\label{obs:compact}
	A push applied to a compact configuration $C$ results in a compatible compact configuration $C'$.
\end{observation}
\begin{proof}
Since $C$ can be obtained from a canonical configuration by pushes, then so can $C'$, and therefore $C'$ is compact, as well.
	It is enough to show that a $\leftpush$ push in $C$ does not change the number of rows (resp., columns) of each given length. Note that labels are irrelevant, and so we will assume $C$ and $C'$ to be unlabeled, without loss of generality.
	We reuse the notation of the proof of Observation~\ref{obs:compact-shape}.
	The $\leftpush$ push only moves tokens in $H$, so it is clear that the lengths of rows do not change.
	We can see the push as removing the leftmost column (which has length $|H|$), moving all columns to the right of $V$ by one unit to the right, and creating a new column of length $|H|$ immediately to the right of $V$.
	Thus, the number of columns of each length remains the same, and $C'$ is compatible with $C$.
\end{proof}

Note that compatibility is an equivalence relation on the set of compact configurations. 
By Observation~\ref{obs:compact}, in any equivalence class of unlabeled compatible compact configurations there is exactly one canonical configuration. Once we reach a compact configuration via a sequence of pushes, no other canonical configuration can be reached, except for the one corresponding to the current compact configuration. 

The following observation follows from the fact that we can obtain a canonical configuration by applying $\downpush$ and $\leftpush$ pushes and by Corollary~\ref{cor:reversible}.

\begin{observation}
	\label{obs:canonical-reach}
Any two unlabeled compatible compact configurations can be obtained from each other by pushes.\qed
\end{observation}

Two labeled configurations are said to have the \emph{same shape} if they have the same set of full positions (recall that we are identifying configurations that only differ by a translation).
We can now define our Permutation Puzzle.

\begin{problem}[Permutation Puzzle]\label{prob:permutation}
	Given two same-shaped labeled compact configurations $C$ and $C'$, is there a sequence of pushes that transforms $C$ into $C'$?
\end{problem}

Two same-shaped labeled configurations differ by a permutation of their tokens.
If we obtain a same-shaped configuration $C'$ from $C$, the sequence of pushes results in such a permutation.
Note that the set of possible permutations is \emph{finite} and closed under composition, and therefore is a \emph{permutation group}, which we denote as $G_C$. 

The following observation allows us to focus on sequences between canonical configurations.
We say that the labeled canonical configuration $C$ obtained from a compact configuration $C'$ by $\langle\, \downpush^{k_1}\ \leftpush^{k_2}\, \rangle$, for some $k_1,k_2$,  is the \emph{canonical form} of $C'$.
Then the following is a direct consequence of Corollary~\ref{cor:reversible}.

\begin{observation}
	\label{obs:canonical-sufficient}
	A labeled compact configuration $C$ can be obtained from a same-shaped labeled compact configuration $C'$ if and only if the canonical form of $C$ is reachable from the canonical form of $C'$.\qed
\end{observation}

\section{Compaction Puzzles}\label{s:3}
In this section, we focus on Compaction Puzzles (Problem~\ref{prob:compaction}) where the input configuration is \emph{sparse}; that is, where no row or column has more than one token. 
We assume that the number of tokens is $n=ab$, and we wonder if there is a sequence of pushes that produces an $a\times b$ box.
Note that, as long as the configuration is sparse, any push shrinks the bounding box, and therefore is irreversible (unless $n=1$).
Recall that without the sparseness assumption, the problem is NP-complete~\cite{trash}.

One may wonder if the leeway given by sparse configurations is enough to obtain every compact configuration; in this section, we will show that this is not the case.

We begin by observing that, for every $n$, there exists a ``universal'' configuration that can be reconfigured into any compact configuration with $n$ tokens (such as an $a\times b$ box).

\begin{observation}
	\label{obs:universal-conf}
	There exists a sparse unlabeled configuration with $n$ tokens that can be reconfigured into any compact configuration with $n$ tokens.
\end{observation}
\begin{proof}
	Such a configuration is the secondary diagonal of an $n\times n$ matrix, i.e., all positions $(i,i)$ are full, for $i\in\{1,\ldots,n\}$, and all other positions are empty. Due to Observation~\ref{obs:canonical-reach}, it suffices to show that any canonical configuration can be formed.
	
	By definition, a canonical configuration has a monotonic decreasing sequence of column lengths.
	After $k$ $\leftpush$ pushes, the configuration will have a column of length $k$.
	We then perform $k$ $\downpush$ pushes bringing the column up aligning its bottom with the next token. 
	Any subsequent $k'\le k$ $\leftpush$ pushes would accumulate tokens in the bottommost contiguous positions of the second leftmost column. 
	We can then repeat this procedure to produce any monotonic decreasing sequence of column lengths.
\end{proof}

We will now show that not all configurations can be reconfigured into an $a\times b$ box.

\begin{lemma}
\label{lem:not-reconfigurable}
For all $a\geq 4$ and $b\geq 3$, there exists a sparse configuration that cannot be reconfigured into an $a\times b$ box.
\end{lemma}
\begin{proof}
We will describe a sparse configuration $C_{a,b}$ and argue that it cannot be reconfigured into an $a\times b$ box; refer to Figure~\ref{fig:counterexample}.
The configuration $C_{a,b}$ is subdivided into four quadrants;
the lower-left and upper-right quadrants are empty, and all tokens are in the other two quadrants.
Specifically, we place $n_1=\left\lfloor\frac{ab}{2}\right\rfloor$ tokens in the upper-left quadrant and $n_2=\left\lceil\frac{ab}{2}\right\rceil$ in the lower-right quadrant; each of those quadrants is a square matrix with only its secondary diagonal full. 

\begin {figure}
  \centering
  \includegraphics [scale=1] {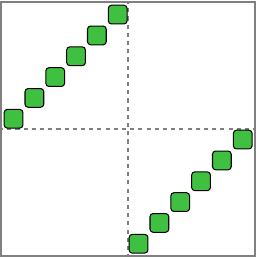}\qquad
  \includegraphics [scale=1] {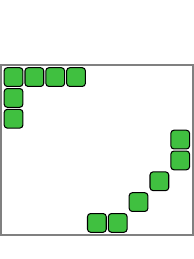}\qquad
  \includegraphics [scale=1] {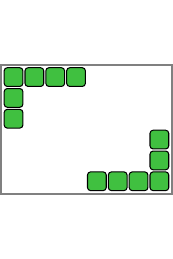}\qquad
  \includegraphics [scale=1] {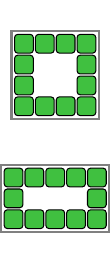}
  \caption {The configuration $C_{a,b}$ defined in Lemma~\ref{lem:not-reconfigurable}, with $a=4$ and $b=3$ (left). Any sequence of pushes produces an ``L-shaped'' pattern (center), which eventually causes the formation of rows with more than $a$ tokens or columns with more than $b$ tokens (right).}
  \label {fig:counterexample}
\end {figure}

Initially, all the $\leftpush$ and $\uppush$ pushes (resp., $\rightpush$ and $\downpush$ pushes) only affect tokens in the upper-left (resp., lower-right) quadrant. Thus, as soon as either $n_1-1$ pushes have been made in the upper-left quadrant or $n_2-1$ pushes have been made in the lower-right quadrant (whichever comes sooner), all the tokens in that quadrant must be located on the union between a single row and a single column, forming a connected ``L-shaped'' pattern.

If the L-shaped pattern contains $x$ tokens on the same row and $y$ tokens on the same column, we must have $x+y=n_i+1$ for some $i\in\{1,2\}$. Thus, $x+y\geq \left\lfloor\frac{ab}{2}\right\rfloor+1$.
It is easy to see (cf.~\cite[Observation~3]{trash}) that it is impossible to form an $a\times b$ box from a configuration where more than $a$ (resp., $b$) tokens are on the same row (resp., column). Therefore, we must have $x\leq a$ and $y\leq b$, which implies $a+b\geq \left\lfloor\frac{ab}{2}\right\rfloor+1 \geq \frac{ab+1}{2}$. By rearranging terms we have $(a-2)(b-2)\leq 3$, whose only solutions with $a\geq 4$ and $b\geq 3$ are $a=4$, $b=3$ and $a=5$, $b=3$.

Let $a=4$ and $b=3$ (resp., $a=5$ and $b=3$), and let $x$ and $y$ be defined as above. Since $x+y\geq 7$ (resp., $x+y\geq 8$), we must have $x=a$ and $y=b$. Now, further pushes in the same quadrant cause the L-shaped pattern to rigidly move toward the opposite quadrant, eventually creating a row with more than $a$ tokens or a column with more than $b$ tokens. On the other hand, making more pushes in the opposite quadrant ends up forming another, oppositely oriented, L-shaped pattern identical to the first. At this point, any push causes the two L-shaped patterns to rigidly approach each other, eventually creating a row or a column with too many tokens.
\end{proof}

In all other cases, any sparse configuration can be reconfigured into an $a\times b$ box, which yields the following theorem.

\begin{theorem}\label{thm:compaction}
	All sparse configurations of $n=ab$ tokens can be pushed into an $a\times b$ box if and only if $a\leq 2$ or $b\leq 2$ or $a=b=3$.
\end{theorem}
\begin{proof}
	Lemma~\ref{lem:not-reconfigurable} provides counterexamples for $a\geq 4$ and $b\geq 3$ (and, symmetrically, for $a\geq 3$ and $b\geq 4$); it remains to show the positive cases.
	
	First, note that the claim for $a=1$ is trivial (and $b=1$ is symmetric): Perform $\downpush$ pushes until all tokens are in the same row, then perform $\leftpush$ pushes until they occupy consecutive positions. 
	
	Now assume $a=2$ ($b=2$ is symmetric).
	Perform $\uppush$ pushes until half of the tokens are in the same row $r$. Then, the other half form a sparse subconfiguration below $r$. Perform $\downpush$ pushes until all such tokens are in the row immediately below $r$.
	Then, a $2\times \frac{n}{2}$ box is obtainable by performing $\leftpush$ pushes.

\begin {figure}
  \centering
  \includegraphics [scale=1] {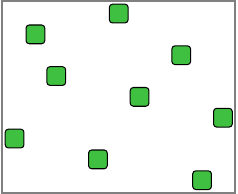}\qquad
  \includegraphics [scale=1] {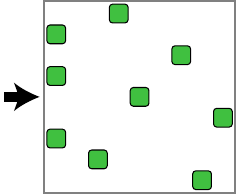}\qquad
  \includegraphics [scale=1] {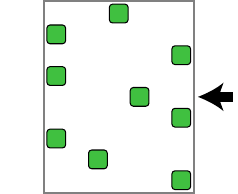}\\
  \vspace{0.5cm}
  \includegraphics [scale=1] {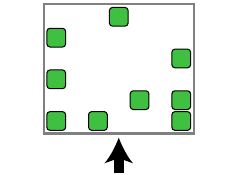}\qquad
  \includegraphics [scale=1] {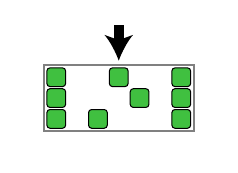}\qquad
  \includegraphics [scale=1] {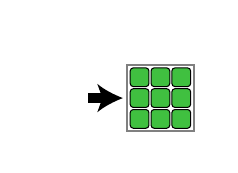}
  \caption {Pushing a sparse configuration of $n=9$ tokens into a $3\times 3$ box (Theorem~\ref{thm:compaction}).}
  \label {fig:3x3}
\end {figure}

	Finally, assume that $a=b=3$; refer to Figure~\ref{fig:3x3}.
	Because the initial configuration is sparse, if we perform $\leftpush$ pushes, the number of tokens in the leftmost column increases by at most one with every push.
	Apply $\leftpush$ pushes until there are three tokens in the leftmost column.
	The subconfiguration obtained by deleting the leftmost column is also sparse.
	Then, by the same argument we can perform $\rightpush$ pushes until there are three tokens in the rightmost column.
	Since we have not yet performed any vertical push, every row has at most one token. 
	There are exactly three tokens that are not in the leftmost or rightmost columns.
	Let $t_1$, $t_2$ and $t_3$, respectively, be those tokens, from smallest to largest $y$-coordinate.
	Perform $\downpush$ pushes until the bottommost row of the configuration is one unit below $t_2$. 
	Note that $t_1$ must be in the bottommost row.
	Symmetrically, apply $\uppush$ pushes until the topmost row of the configuration is one unit above $t_2$. 
	Now, $t_1$, $t_2$ and $t_3$ are strictly between the leftmost and rightmost columns, and are one in each row.
	Then, we obtain a $3\times 3$ box by performing $\leftpush$ pushes.
\end{proof}

\section{Permutation Puzzles}\label{s:4}
In this section, we will give a complete solution to the Permutation Puzzle (Problem~\ref{prob:permutation}): For any given starting compact labeled configuration $C$, we will determine the set of labeled configurations of the same shape as $C$ that we can reach by means of pushes. Specifically, we will give a description of the permutation group $G_C$, as defined in Section~\ref{s:2}.

By Observation~\ref {obs:canonical-sufficient}, it is not restrictive to limit our attention to canonical configurations. For most of this section, we will assume that all tokens have distinct labels; the general case will be discussed at the end of the section.

We will show that all feasible permutations in this setting are even (Section~\ref{s:4.1}), and furthermore that all even permutations are feasible, with some simple restrictions (Section~\ref{s:4.2}).

It will be convenient to always consider the lower-left corner of the bounding box to be $(0,0)$, even after a sequence of pushes. That is, we consider the ``gravity'' formulation of the puzzle, where the bounding box remains still and the tokens move within it (recall that compact configurations are uncompressible). As a consequence, a push in one direction will cause the tokens to ``fall'' in the opposite direction within the bounding box.

\subsection{All Feasible Permutations Are Even}\label{s:4.1}

In this section, we will prove Theorem~\ref{thm:even}, which states that only \emph {even} permutations are possible in the Permutation Puzzle. 

For a labeled canonical configuration $C$, let $C'$ be an extension of the labeling where also the empty cells inside the bounding box of $C$ get a unique label.
Our proof strategy is to first extend permutations of $C$ to permutations of $C'$
, and argue that a permutation of full and empty cells must be even.
We then introduce a \emph {dual game} played on the empty cells only, and argue that this dual game has similar properties.
Since the dual of any game is always smaller (in terms of bounding box) than the original, our theorem then follows by induction.

\subsubsection* {Permutations on Full Cells and Empty Cells}

Let $C$ be a labeled canonical configuration; that is, a function $C\colon \mathcal{L} \to \Sigma \cup \{{\rm empty}\}$.
We extend $C$ to another function $C'\colon \mathcal{L} \to \Sigma \cup \Sigma' \cup \{{\rm empty}\}$, where $\Sigma'$ is a second set of unique labels and $C'(x) \in \Sigma'$ if and only if $x$ is in the bounding box of $C$, but not in $C$ (for all other $x$, $C'(x)=C(x)$).

We now define the effect of a push operation on $C'$ (illustrated in Figure~\ref {fig:empty-labels}). We define it for a $\rightpush$ push; the other directions are symmetric. A single $\rightpush$ push affects each row as follows:
\begin {itemize}
  \item For each row in which the rightmost cell is \emph {empty}, we shift all tokens and empty cells one position to the right and place the rightmost empty cell at the left; in other words, we perform a single cyclic permutation on the tokens in the row.
  \item For each row in which the rightmost cell is \emph {full}, nothing changes.
\end {itemize}

\begin {figure}
\centering
  \includegraphics [scale=0.94] {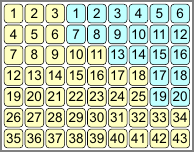}
  \includegraphics [scale=0.94] {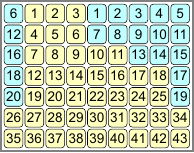}
  \includegraphics [scale=0.94] {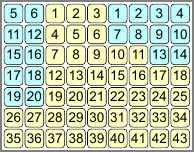}
  \includegraphics [scale=0.94] {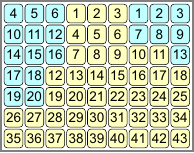}
  \includegraphics [scale=0.94] {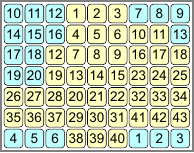}
  \caption {A configuration with labeled empty cells and the result after the sequence of pushes $\langle\, \rightpush\ \rightpush\ \rightpush\ \uppush\, \rangle$. We use yellow for full tokens and light blue for empty cells (and later, dual tokens) in this section.}
  \label {fig:empty-labels}
\end {figure}

We now argue that for any sequence of pushes which transforms $C$ into another canonical configuration, the effect of these moves on $C'$ must be an even permutation.

\begin {lemma} \label {lem:all-even}
Every achievable permutation on the union of full and empty cells must be even.
\end {lemma}
\begin {proof}
By definition, every horizontal (resp., vertical) push causes a cyclic permutation on some rows (resp., on some columns) involving both labeled tokens and labeled empty cells.
We will argue that the total number of cyclic permutations on the rows (resp., on the columns) caused by these pushes is even. Since all cyclic permutations on rows (resp., on columns) have the same parity, which depends only on the width (resp., height) of the bounding box, this implies that the overall permutation is even.

\begin {figure}[h!]
\centering
  \includegraphics [scale=1] {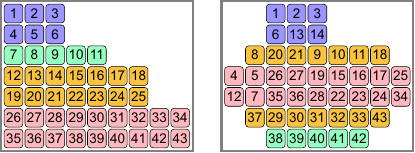}
\qquad
  \includegraphics [scale=1] {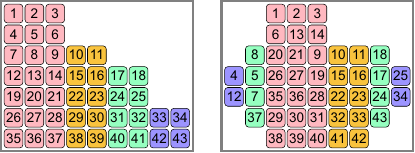}
  \caption {Rows and columns of the same length are always aligned.}
  \label {fig:aligned}
\end {figure}

From Observations~\ref{obs:compact-shape} and~\ref{obs:compact} it follows that all rows (resp., columns) of the same length must always be aligned throughout the reconfiguration (see Figure~\ref {fig:aligned}).
In particular, observe that a \emph {vertical} push never influences the \emph {horizontal} placement of the rows of a particular length (note that here we only argue about the shape, not the labels).
Now consider the set of rows of length $k$. Every horizontal push either moves all such rows one cell to the right, or one cell to the left, or not at all. Since both at the start and at the end of the process all rows of length $k$ are aligned with the left border of the bounding box, the total number of pushes that influence the horizontal placement of these rows is even, and thus the number of cyclic permutations performed on these rows is also even.

Note that a single push may influence the placement of rows of different lengths; however, for each specific length $k$, the total number of pushes that influences them is even, and therefore the total number of cyclic permutations on rows is even.
The same holds for the cyclic permutations on columns, which concludes the proof.
\end {proof}

Now, in order to prove our main theorem, we still need to show that the resulting permutation \emph {restricted to the tokens} is also even; see Figure~\ref {fig:restrict}.

\begin {figure}
\centering
  \includegraphics [scale=1.25] {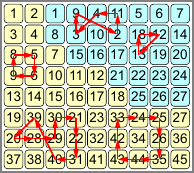}
  \caption {The fact that the overall permutation on full and empty cells is even does not necessarily imply that the permutation on the full cells is even.}
  \label {fig:restrict}
\end {figure}

\subsubsection* {Dual Puzzles}

Given a configuration $C$, we have extended its labeling to a configuration $C'$ having labels on the empty cells, as well.
Now, consider the \emph {restriction} of $C'$ to \emph {only} the empty cells; that is, the function $D \colon \mathcal{L} \to \Sigma'\cup \{{\rm empty}\}$ which labels exactly the cells that are empty in $C$ but lie inside the bounding box of $C$. Clearly, if we can prove that a permutation on $D$ is even, then Lemma~\ref {lem:all-even} implies that the corresponding permutation on $C$ must also be even (the product of two permutations is even if and only if they have the same parity).

\begin {figure}
  \centering
  \includegraphics [scale=1] {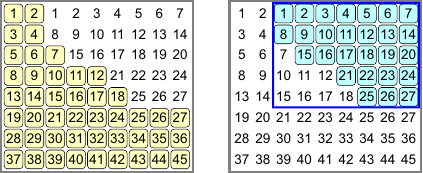}
  \qquad
  \includegraphics [scale=1] {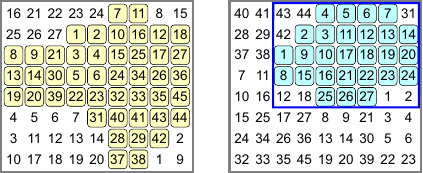}
  \caption {A primal puzzle (yellow) with its corresponding dual puzzle (light blue). The blue rectangle represents the bounding box of the dual puzzle. Note that the tokens in the dual puzzle match the labels on the empty cells of the primal puzzle.}
  \label {fig:dual}
\end {figure}

For convenience, we will consider the bounding box of $C$ as a torus and display it in such a way that all full rows and columns are aligned with the left and bottom of the rectangle (see Figure~\ref {fig:dual}).
We call $D$ a \emph {dual configuration}, and we will study the effects of push moves on $D$, which we will treat as a \emph {dual puzzle} (while the original puzzle on $C$ is the \emph{primal puzzle}).

\begin {figure}
  \centering
  \includegraphics [scale=1] {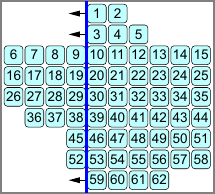}\quad
  \includegraphics [scale=1] {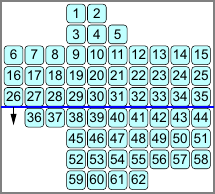}\quad
  \includegraphics [scale=1] {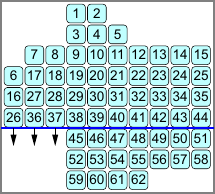}\quad
  \includegraphics [scale=1] {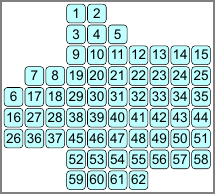}
  \caption {Rules of the dual puzzle. 
  The sequence of pulls $\langle\, \rightpush\ \uppush\ \uppush\, \rangle$ is shown.}
  \label {fig:dual-rules}
\end {figure}

When considering a dual puzzle in isolation, we swap terminology and refer to the empty cells as full, and vice versa.
In this context, a \emph {push} move in the primal puzzle either leaves the dual puzzle unchanged or causes a \emph {pull} move in the dual puzzle. Specifically, if we perform a $\rightpush$ push move in the primal puzzle from a configuration where only the full rows are touching the right side of the bounding box, nothing happens in the dual puzzle. Otherwise, the $\rightpush$ push move in the primal puzzle causes a $\rightpush$ pull move in the dual puzzle, where exactly those rows that are farthest from the left boundary are pulled one unit to the left (refer to Figure~\ref{fig:dual-rules}, where the blue line represents a side of the bounding box of the primal puzzle).

Thus, as we play the primal puzzle, we are also playing the dual puzzle.
We say that a \emph {dual} configuration is \emph {canonical} when all tokens are aligned with the top and right borders of their bounding box; this way, the empty cells in a canonical (primal) configuration form a canonical (dual) configuration.
Most of the results obtained so far for primal puzzles automatically apply to dual puzzles, as well.

In particular, we claim that in the dual game, the equivalent of Lemma~\ref {lem:all-even} still holds.
For this, we now consider again an extension of our dual configuration $D$ which labels both the full and empty cells of the bounding box of $D$.
Crucially, this is \emph {not} the same as the original extension $C'$, because the bounding box of $D$ is smaller than the bounding box of $C$.

\begin {lemma} \label {lem:all-even-dual}
  Let $D$ be a labeled canonical configuration, and let $D'$ be an extension of $D$ that labels also the empty cells of the bounding box of $D'$. Every achievable permutation on $D'$ under a sequence of \emph {dual moves} must be even.
\end {lemma}

\begin {proof}
Since the number of rows (resp., columns) of any given length in the primal puzzle remains constant, then the same is true in the dual puzzle. Moreover, all of the rows (resp., columns) of the same length must always be aligned in the primal puzzle, and therefore they must be aligned in the dual puzzle, as well. The proof now proceeds exactly as in Lemma~\ref{lem:all-even}.
\end {proof}

\subsubsection* {Induction}

We are now ready to prove that all permutations in $G_C$ must be even.

\begin {figure}
\centering
  \includegraphics [scale=1] {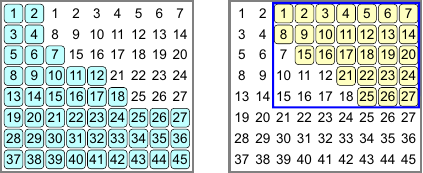}
\qquad
  \includegraphics [scale=1] {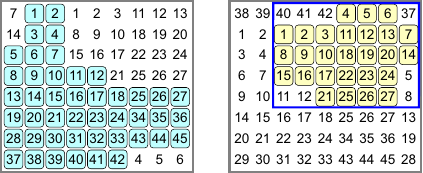}
  \caption {The dual of a dual puzzle is again a primal puzzle; we see the result after the sequence of pulls $\langle\, \leftpush\ \downpush\, \rangle$.}
  \label {fig:dual-dual}
\end {figure}





\begin {theorem} \label {thm:even}
Let $C$ be a labeled canonical configuration, and let $\pi\in G_C$. Then, $\pi$ is an even permutation.
\end {theorem}
\begin {proof}
The proof follows from two simple observations. Firstly, if we take the dual of a dual-type puzzle, we obtain another puzzle that again follows the rules of a primal-type puzzle (refer to Figure~\ref {fig:dual-dual}). Secondly, the bounding box of a (primal or dual) puzzle is strictly larger than the bounding box of its dual.

We will prove a stronger statement: that our theorem holds for both primal-type and dual-type puzzles. The proof is by well-founded induction on the size of the bounding box.

If the bounding box of our puzzle is completely full, then moves have no effect, and the only allowed permutation is the identity, which is even.
  
Now, consider a canonical configuration $C$ in a (primal or dual) puzzle with a bounding box which is not completely full. Such a puzzle has a dual, with a canonical configuration $D$ corresponding to $C$. The bounding box of $D$ has smaller size, and therefore the induction hypothesis applies to the dual puzzle. After performing some moves and restoring a canonical configuration in both puzzles, the tokens in $C$ have undergone a permutation $\pi\in G_C$, while the tokens in $D$ have undergone a permutation $\sigma\in G_D$. By Lemmas~\ref {lem:all-even} and~\ref {lem:all-even-dual}, the overall permutation $\pi\sigma$ is even; by the induction hypothesis, $\sigma$ is even; hence, $\pi$ is even, as well.
\end {proof}
\subsection{Generating All Feasible Permutations}\label{s:4.2}

In this section, we will give a complete description of the permutation group $G_C$, which we already know from Section~\ref{s:4.1} to be a subgroup of the \emph{alternating group} $\Alt{n}$.

\subsubsection*{Unmovable Central Core}
Let the bounding box be an $a\times b$ rectangle. Our first observation is that, if more than half of the rows and more than half of the columns of the bounding box are full, then there is a central box of tokens that cannot be moved.
Let $a'$ (resp., $b'$) be the number of full columns (resp., full rows), and let $a''=a-a'$ and $b''=b-b'$. 

\begin{definition}[Core]
If $a'>a''$ and $b'>b''$, the \emph{core} of $C$ is the set of lattice points in the bounding box of $C$ that are in the central $a'-a''$ columns and in the central $b'-b''$ rows.\footnote{Equivalently, if $a>2a''$ and $b>2b''$, the core is obtained by discarding the $a''$ leftmost columns, the $a''$ rightmost columns, the $b''$ topmost rows, and the $b''$ bottommost rows.} If $a'\leq a''$ or $b'\leq b''$, the \emph{core} of $C$ is empty.
\end{definition}

The lattice points in the core are called \emph{core points}, and the tokens in core points are called \emph{core tokens}. Figure~\ref{fig:core} shows an example of a non-empty core.

\begin {figure}[t]
  \centering
  \includegraphics [scale=1.25] {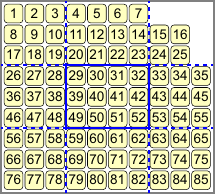}
  \caption {A configuration with $a=10$, $b=9$, $a'=7$, and $b'=6$. The blue rectangle in the center surrounds the core tokens.}
  \label {fig:core}
\end {figure}

\begin{observation}\label{obs:core}
No permutation in $G_C$ moves any core token.
\end{observation}
\begin{proof}
If the core is empty, there is nothing to prove; hence, let us assume that $a'>a''$ and $b'>b''$. From Section~\ref{s:2}, we know that there are always exactly $a'$ contiguous full columns and $b'$ contiguous full rows, no matter how the tokens are pushed. Hence, the central $a-2a''=a'-a''$ columns (resp., the central $b-2b''=b'-b''$ rows) are always full, and are not affected by $\uppush$ or $\downpush$ pushes (resp., $\leftpush$ or $\rightpush$ pushes). Therefore, no push can affect the core tokens.
\end{proof}

\subsubsection*{Permutation Groups}
In order to understand the structure of $G_C$, we will review some notions of group theory and prove three technical lemmas.

\begin{definition}
A permutation group $G$ on $\{1,2,\dots,n\}$ is \emph{2-transitive} if, for every $1\leq x,y,w,z\leq n$ with $x\neq y$ and $w\neq z$, there is a permutation $\pi\in G$ such that $\pi(x)=w$ and $\pi(y)=z$.
\end{definition}

\begin{theorem}[Jones, \cite{jones}]\label{theorem-jones}
If $G$ is a 2-transitive permutation group on $\{1,2,\dots,n\}$ and $G$ contains a cycle of length $n-3$ or less, then $G$ contains all even permutations.\footnote{Jones' theorem in \cite{jones} holds more generally for \emph{primitive} permutation groups. The fact that all 2-transitive permutation groups are primitive is an easy observation.}\qed
\end{theorem}

To express cyclic permutations, we use the standard notation $\sigma=(s_1\ s_2\ s_3\ \dots\ s_k)$, occasionally adding commas between terms when doing so improves readability. The cycle $\sigma$ is the permutation that fixes all items except $s_1$, $s_2$, $\dots$, $s_k$ such that $\sigma(s_1)=s_2$, $\sigma(s_2)=s_3$, $\dots$, $\sigma(s_k)=s_1$. Since we are studying permutations puzzles, it is visually more convenient to interpret a permutation as acting on places rather than on items. Thus, for example, $(1\ 2\ 3)$ is understood as the cycle involving the tokens occupying the \emph{locations} labeled $1$, $2$, and $3$ rather than the \emph{tokens} labeled $1$, $2$, and $3$. Also, we will follow the convention to compose chains of permutations from left to right, which is the common one in permutation theory.

\begin{lemma}\label{lemma-perm1}
Let $\alpha=(1,2,\dots,a)$ and $\beta=(a-b+1,a-b+2,\dots,2a-b)$ be two cycles spanning $n=2a-b$ items, with $a\geq 2$ and $1\leq b<a$. Then, the permutation group generated by $\alpha$ and $\beta$ acts 2-transitively on $\{1,2,\dots,n\}$.
\end{lemma}
\begin{proof}
Let $G$ be the group generated by $\alpha$ and $\beta$, and let $A$ (resp., $B$) be the set of items spanned by $\alpha$ (resp., $\beta$). Observe that $|A|=|B|=a$, $|A\cap B|=b$, and $|A\setminus B|=|B\setminus A|=a-b>0$. Assume that $w=a-b$ and $z=n$; we will prove that, for all $1\leq x,y\leq n$ with $x\neq y$, there is a permutation $\pi\in G$ such that $\pi(x)=w$ and $\pi(y)=z$. This will be sufficient to conclude that $G$ acts 2-transitively on $\{1,2,\dots,n\}$. Indeed, let $1\leq x,y,w', z'\leq n$ with $x\neq y$ and $w'\neq z'$. Since we have two permutations $\pi_1,\pi_2\in G$ such that $\pi_1(x)=w$,  $\pi_1(y)=z$, $\pi_2(w')=w$,  $\pi_2(z')=z$, then the permutation $\pi_1\pi_2^{-1}\in G$ maps $x$ to $w'$ and $y$ to $z'$, respectively.

Assume first that $x\in A$ and $y\in B$. Let $0\leq d<a$ be such that $\alpha^d(x)=w$, and let $0\leq d'<a$ be such that $\beta^{d'}(y)=z$. For symmetry reasons, we may assume without loss of generality that $d\leq d'$. Then, it is easy to see that $\alpha^d(y)=y'\in B$. Let $0\leq d''<a$ be such that $\beta^{d''}(y')=z$. Now, if we set $\pi=\alpha^d\beta^{d''}$, we have $\pi(x)=(\alpha^d\beta^{d''})(x)=\beta^{d''}(w)=w$ and $\pi(y)=(\alpha^d\beta^{d''})(y)=\beta^{d''}(y')=z$, as desired.

Assume now that $x\in B$ and $y\in A$. We can construct a permutation $\tau$ such that $\tau(x)=z$ and $\tau(y)=w$ as we did above. Then, we define $\sigma=\alpha^{-1}\beta^{-1}\alpha\beta^2$ if $b=1$ and $\sigma=\alpha^{-1}\beta^{-1}\alpha\beta$ if $b>1$. It is easy to see that $\sigma(w)=z$ and $\sigma(z)=w$. Let $\pi=\tau\sigma$; we have $\pi(x)=(\tau\sigma)(x)=\sigma(z)=w$ and $\pi(y)=(\tau\sigma)(y)=\sigma(w)=z$.

Finally, assume that $x,y\in A\setminus B$ (the case where $x,y\in B\setminus A$ is symmetric). We can iterate $\alpha$ until either $x$ or $y$ (say $y$) is mapped to $A\cap B$. At this point, we are in a situation where $x\in A$ and $y\in B$, which we already solved.
\end{proof}

\begin {figure}[t]
  \centering
  \includegraphics [scale=1] {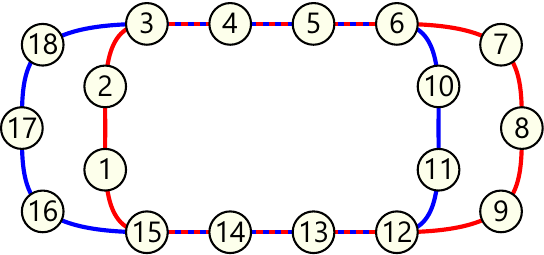}
  \caption {Example of the two cycles in Lemma~\ref{lemma-perm2} with $a=2$ and $b=4$.}
  \label {fig:2trans-1}
\end {figure}

\begin{lemma}\label{lemma-perm2}
Let $\alpha=(1,\dots,2a+b+1,3a+b+2,\dots,3a+2b+1)$ and $\beta=(a+1,\dots,a+b,2a+b+2,\dots,4a+2b+2)$ be two cycles spanning $n=4a+2b+2$ items, with $a\geq 0$ and $b\geq 1$. Then, the permutation group generated by $\alpha$ and $\beta$ acts 2-transitively on $\{1,2,\dots,n\}$.
\end{lemma}
\begin{proof}
The two cycles are illustrated in Figure~\ref{fig:2trans-1}. Let $G$ be the group generated by $\alpha$ and $\beta$, and let $A$ (resp., $B$) be the set of items spanned by $\alpha$ (resp., $\beta$). Assume that $x=n/2$ and $y=n$; we will prove that, for all $1\leq w,z\leq n$ with $w\neq z$, there is a permutation $\pi\in G$ such that $\pi(x)=w$ and $\pi(y)=z$. This suffices to conclude that $G$ acts 2-transitively on $\{1,2,\dots,n\}$, in a similar way to the previous lemma.

Assume first that $w,z\in A$, and let $0< d< |A|=2a+2b+1$ be such that $\alpha^d(w)=z$ (i.e., $d$ is the ``distance'' from $w$ to $z$ along $\alpha$). We construct a permutation $\tau$ as follows:
\begin{itemize}
\item If $1\leq d\leq a$ (and $a>0$), we set $\tau=\alpha^{a+b+1-d}\beta$.
\item If $a+1\leq d\leq a+b-1$ (and $b>1$), we set $\tau=\alpha^{a+b-d}\beta$.
\item If $d=a+b$, we set $\tau=\alpha^{-a}\beta^{-a-1}\alpha^{a+1}$.
\item If $a+b+1\leq d\leq 2a+b+1$, we set $\tau=\alpha^{a+b+1-d}\beta$.
\item If $2a+b+2\leq d\leq 2a+2b$ (and $b>1$), we set $\tau=\alpha^{a+b-d}\beta$.
\end{itemize}
It is straightforward to check that $\tau(y)=a+1=(\tau\alpha^d)(x)$; that is, $\tau$ places $y$ in position $a+1$ and places $x$ at the correct distance from $y$ along $\alpha$. Now we can set $\pi=\tau\alpha^i$, where $\alpha^i(a+1)=z$, so that $\pi(y)=z$ and $\pi(x)=(\tau\alpha^i)(x)=(\tau\alpha^d\alpha^i\alpha^{-d})(x)=(\alpha^i\alpha^{-d})(a+1)=\alpha^{-d}(z)=w$.

The case with $w,z\in B$ is symmetric and will be omitted.

Assume now that $w\in A\setminus B$ and $z\in B\setminus A$. Therefore, we can set $\pi=\alpha^i\beta^j$, where $\alpha^i(x)=w$ and $\beta^j(y)=z$. Since $y\notin A$, we have $\alpha^i(y)=y$; since $w\notin B$, we have $\beta^j(w)=w$. We conclude that $\pi(x)=w$ and $\pi(y)=z$.

Finally, assume that $w\in B\setminus A$ and $z\in A\setminus B$. We first construct a permutation $\rho$ that swaps $x$ and $y$ as follows:
\begin{itemize}
\item If $b=1$, we set $\rho=\alpha\beta\alpha^{a+1}\beta^a$.
\item If $b>1$, we set $\rho=\alpha\beta\alpha^{b-2}\beta\alpha^{a+1}\beta^a$.
\end{itemize}
It is straightforward to check that $\rho(x)=\rho(n/2)=n=y$ and $\rho(y)=\rho(n)=n/2=x$. Now that $x$ and $y$ are on the correct cycles, we can proceed as in the previous case.
\end{proof}

\begin{lemma}\label{lemma-perm3}
Let $A$ and $B$ be two finite sets such that $A\cap B\neq \emptyset$ and $A\setminus B\neq\emptyset$. Let $G$ be a permutation group on $A\cup B$ whose restriction to $A$ is 2-transitive and whose restriction to $B\setminus A$ is trivial, and let $\beta$ be a cycle spanning $B$. Then, the permutation group generated by $G$ and $\beta$ acts 2-transitively on $A\cup B$.
\end{lemma}
\begin{proof}
Let us fix $x\in A\setminus B$ and $y\in A\cap B$. For all $w,z\in A\cup B$ with $w\neq z$, we will prove that there is a permutation $\pi$ in the group generated by $G$ and $\beta$ such that $\pi(x)=w$ and $\pi(y)=z$. As in the previous lemmas, this is sufficient to conclude that the group acts 2-transitively on $A\cup B$.

If $w,z\in A$, there is a permutation $\pi\in G$ that maps $x$ to $w$ and $y$ to $z$, because $x,y\in A$ and $G$ acts 2-transitively on $A$.

Assume now that $w,z\in B$, and let $0< d< |B|$ be such that $\beta^d(w)=z$. Let $y'=\beta^d(y)$, and let $\tau\in G$ be a permutation such that $\tau(x)=y$ and $\tau(y')=y'$, which exists because $G$ acts 2-transitively on $A$ and fixes $B\setminus A$. We can set $\pi=\beta^d\tau\beta^i$, where $\beta^i(y')=z$. Clearly, $\pi(y)=(\beta^d\tau\beta^i)(y)=(\tau\beta^i)(y')=\beta^i(y')=z$, and $\pi(x)=(\beta^d\tau\beta^i)(x)=(\tau\beta^i)(x)=\beta^i(y)=\beta^{i-d}(y')=\beta^{-d}(z)=w$.

Let us now consider the case where $w\in A\setminus B$ and $z\in B\setminus A$. We can set $\pi=\beta^j\rho$, where $\beta^j(y)=z$ and $\rho\in G$ such that $\rho(x)=w$. It is easy to verify that $\pi(x)=(\beta^j\rho)(x)=\rho(x)=w$ and $\pi(y)=(\beta^j\rho)(y)=\rho(z)=z$.

Finally, if $w\in B\setminus A$ and $z\in A\setminus B$, we can swap $x$ and $y$ via a permutation in $G$, and then proceed as in the previous case.
\end{proof}

\subsubsection*{Generating Cycles}
We will now assume that all labels are distinct. Specifically, the $n$ tokens in $C$ are labeled $1$ to $n$ from left to right and from top to bottom, as in Figure~\ref{fig:core}. The general case will be discussed later.

We introduce three types of sequences of moves:
\begin{itemize}
\item Type-A $k$-sequence: $\langle\, \rightpush^k\ \rightpush\ \uppush\ \leftpush\ \downpush\ \leftpush^k\, \rangle$ for $0\leq k< a''$.
\item Type-B $k$-sequence: $\langle\, \uppush^k\ \uppush\ \rightpush\ \downpush\ \leftpush\ \downpush^k\, \rangle$ for $0\leq k< b''$.
\item Type-C $k$-sequence: $\langle\, \rightpush^k\ \uppush\ \rightpush\ \downpush\ \leftpush\ \leftpush^k\, \rangle$ for $0\leq k< a''$.
\end{itemize}

It is straightforward to check that a type-A $k$-sequence always produces a cycle of length $2a'+2b'-1$ (refer to Figure~\ref{fig:cyclemoves-1}), which involves the lattice points $(k,i)$ and $(a'+k,i)$ for all $0\leq i< b'$ (plus some other points in the rows $(0,\cdot)$ and $(b',\cdot)$). Such cycles are called \emph{type-A cycles}.

\begin {figure}
  \centering
  \includegraphics [scale=1] {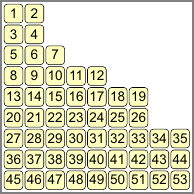}\qquad
  \includegraphics [scale=1] {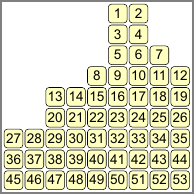}\qquad
  \includegraphics [scale=1] {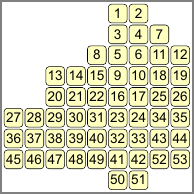}\\
  \vspace{0.5cm}
  \includegraphics [scale=1] {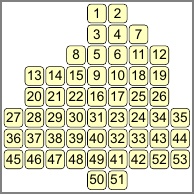}\qquad
  \includegraphics [scale=1] {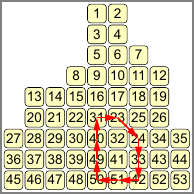}\qquad
  \includegraphics [scale=1] {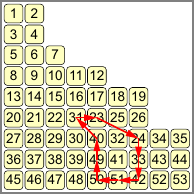}
  \caption {Performing a type-A $k$-sequence: $\langle\, \rightpush^{k+1}\ \uppush\ \leftpush\ \downpush\ \leftpush^k\, \rangle$. Recall that a push in a direction causes the tokens to ``fall'' in the opposite direction within the bounding box.}
  \label {fig:cyclemoves-1}
\end {figure}

Symmetrically, a type-B $k$-sequence produces a \emph{type-B cycle} of length $2a'+2b'-1$ involving (among others) the lattice points $(i,k)$ and $(i,b'+k)$ for all $0\leq i< a'$ (see Figure~\ref{fig:cyclemoves-2}). Thus, the type-A and the type-B cycles collectively cover all the non-core points that lie in one of the $a'$ full columns or in one of the $b'$ full rows (see Figures~\ref{fig:cycles-1} and~\ref{fig:cycles-2}).

\begin {figure}[h!]
  \centering
  \includegraphics [scale=1] {pics/m00.pdf}\qquad
  \includegraphics [scale=1] {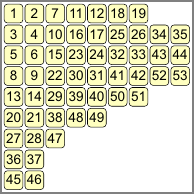}\qquad
  \includegraphics [scale=1] {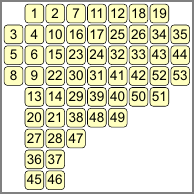}\\
  \vspace{0.5cm}
  \includegraphics [scale=1] {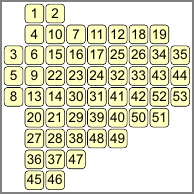}\qquad
  \includegraphics [scale=1] {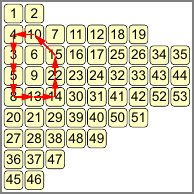}\qquad
  \includegraphics [scale=1] {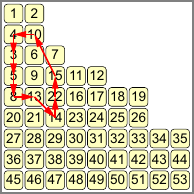}
  \caption {Performing a type-B $k$-sequence: $\langle\, \uppush^{k+1}\ \rightpush\ \downpush\ \leftpush\ \downpush^k\, \rangle$.}
  \label {fig:cyclemoves-2}
\end {figure}

A type-C $k$-sequence produces a cycle, as well, which we call \emph{type-C cycle} (see Figure~\ref{fig:cyclemoves-3}). However, unlike previous cycles, a type-C cycle's length may vary depending on $k$. It is easy to see that, if the row $(\cdot,i)$, with $0\leq i<b$, has length $\ell_i$, then the cycle produced by a type-C $k$-sequence with $a-\ell_i\leq k<a''$ involves (among others) the lattice point $(\ell_i+k-a'',i)$. Such a point lies at distance $a''-k-1$ from the rightmost full token in the row. Thus, the type-C cycles collectively cover all the tokens that are not in a full column (see Figure~\ref{fig:cycles-3}). We conclude that type-A, type-B, and type-C tokens collectively cover all non-core points (see Figure~\ref{fig:cycles-4}).

\begin {figure}
  \centering
  \includegraphics [scale=1] {pics/m00.pdf}\qquad
  \includegraphics [scale=1] {pics/m01.pdf}\qquad
  \includegraphics [scale=1] {pics/m02.pdf}\\
  \vspace{0.5cm}
  \includegraphics [scale=1] {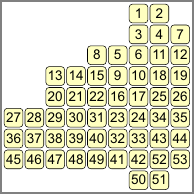}\qquad
  \includegraphics [scale=1] {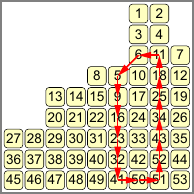}\qquad
  \includegraphics [scale=1] {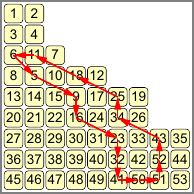}
  \caption {Performing a type-C $k$-sequence: $\langle\, \rightpush^k\ \uppush\ \rightpush\ \downpush\ \leftpush^{k+1}\, \rangle$.}
  \label {fig:cyclemoves-3}
\end {figure}

The next theorem states that, in most cases, the group $G_C$ is exactly the alternating group on the non-core tokens, i.e., the group of all even permutations on the $n-(a'-a'')(b'-b'')$ tokens not in the core (or on all $n$ tokens if the core is empty).
\begin{theorem}\label{theorem-gen1}
If $n/2\geq a'+b'+1$ and $a''b''>1$, then $G_C$ is the alternating group on the non-core tokens.
\end{theorem}
\begin{proof}
We know from Theorem~\ref{thm:even} that $G_C$ only contains even permutations; also, by Observation~\ref{obs:core}, no permutation in $G_C$ can move any core tokens. Hence, it suffices to show that $G_C$ contains \emph{all} even permutations of the non-core tokens. By applying a reflection to $C$ if necessary, we may assume that $a''\geq b''$. Also, since $a''b''>1$, we have $a''\geq 2$, and therefore there are at least two distinct type-A cycles.

Let $\alpha$ and $\beta$ be the type-A cycles for $k=0$ and $k=1$, respectively. Observe that, if $a'=1$, then $\alpha$ and (the inverse of) $\beta$ satisfy the hypotheses of Lemma~\ref{lemma-perm1}; if $a'>1$, then $\alpha$ and $\beta$ satisfy the hypotheses of Lemma~\ref{lemma-perm2}. In both cases, the group $G_1\leq G_C$ generated by $\alpha$ and $\beta$ acts 2-transitively on the points spanned by $\alpha$ and $\beta$ and acts trivially on all other points.

The group $G_1$, together with the type-A cycle for $k=2$, satisfies the assumptions of Lemma~\ref{lemma-perm3}. Therefore, $G_C$ has a subgroup $G_2$ that acts 2-transitively on the points spanned by the first three type-A cycles and acts trivially on all other points. By repeatedly applying Lemma~\ref{lemma-perm3} to all remaining type-A cycles, we conclude that $G_C$ has a subgroup that acts 2-transitively on a set that includes all the non-core points in the $b'$ full rows of $C$.

By applying Lemma~\ref{lemma-perm3} again to the type-B cycles (the first of which properly intersects all of the full rows), we obtain a subgroup of $G_C$ that acts 2-transitively on a set that includes all the non-core points in the $a'$ full columns and in the $b'$ full rows. Finally, we can apply Lemma~\ref{lemma-perm3} to the type-C cycles (all of which properly intersect the full rows) to obtain a subgroup of $G_C$ that acts 2-transitively on all non-core tokens. This implies that $G_C$ itself acts 2-transitively on all non-core tokens, as well.

To conclude the proof, we recall that each type-A cycle has length $2a'+2b'-1$. Since $n/2\geq a'+b'+1$, these cycles have length at most $n-3$. Thus, $G_C$ contains all even permutations of the non-core tokens, due to Theorem~\ref{theorem-jones}.
\end{proof}

\begin {figure}[h!]
  \centering
  \includegraphics [scale=1] {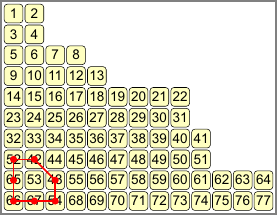}\qquad
  \includegraphics [scale=1] {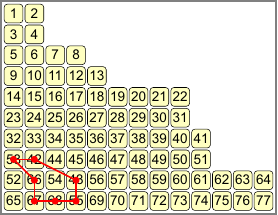}\qquad
  \includegraphics [scale=1] {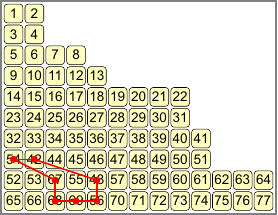}\\
  \vspace{0.5cm}
  \includegraphics [scale=1] {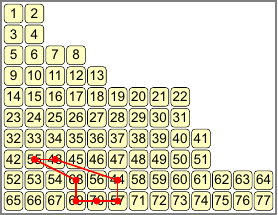}\qquad
  \includegraphics [scale=1] {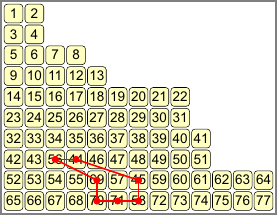}\qquad
  \includegraphics [scale=1] {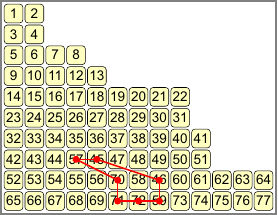}\\
  \vspace{0.5cm}
  \includegraphics [scale=1] {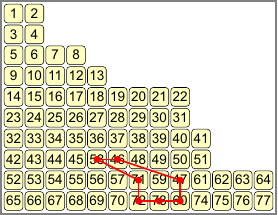}\qquad
  \includegraphics [scale=1] {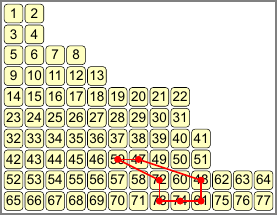}\qquad
  \includegraphics [scale=1] {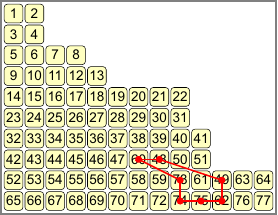}\\
  \vspace{0.5cm}
  \includegraphics [scale=1] {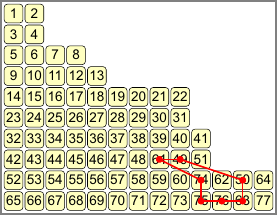}\qquad
  \includegraphics [scale=1] {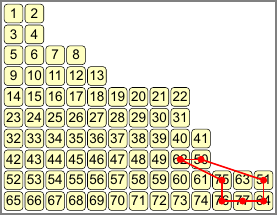}
  \caption {The type-A cycles span, among others, all the (non-core) tokens in the $b'$ full rows.}
  \label {fig:cycles-1}
\end {figure}

\newpage

\begin {figure}[h!]
  \centering
  \includegraphics [scale=1] {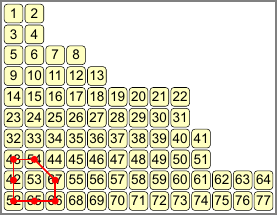}\qquad
  \includegraphics [scale=1] {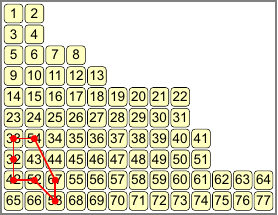}\qquad
  \includegraphics [scale=1] {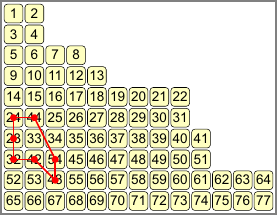}\\
  \vspace{0.5cm}
  \includegraphics [scale=1] {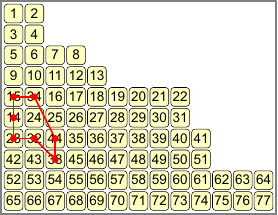}\qquad
  \includegraphics [scale=1] {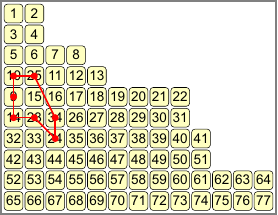}\qquad
  \includegraphics [scale=1] {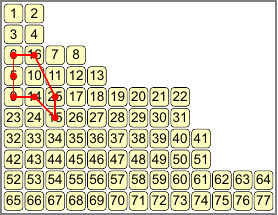}\\
  \vspace{0.5cm}
  \includegraphics [scale=1] {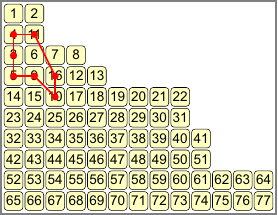}\qquad
  \includegraphics [scale=1] {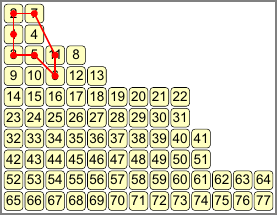}
  \caption {The type-B cycles span, among others, all the (non-core) tokens in the $a'$ full columns.}
  \label {fig:cycles-2}
\end {figure}

\newpage

\begin {figure}[h!]
  \centering
  \includegraphics [scale=1] {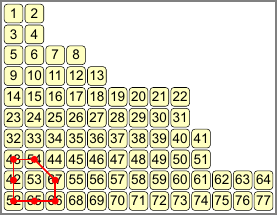}\qquad
  \includegraphics [scale=1] {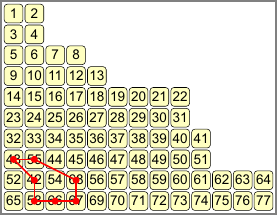}\qquad
  \includegraphics [scale=1] {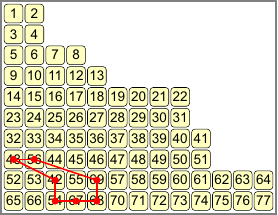}\\
  \vspace{0.5cm}
  \includegraphics [scale=1] {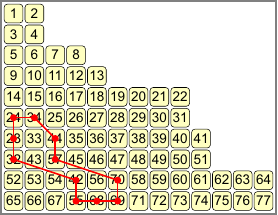}\qquad
  \includegraphics [scale=1] {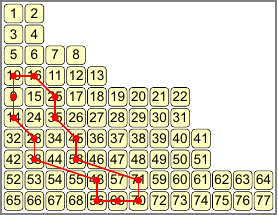}\qquad
  \includegraphics [scale=1] {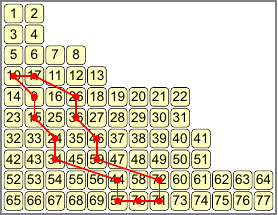}\\
  \vspace{0.5cm}
  \includegraphics [scale=1] {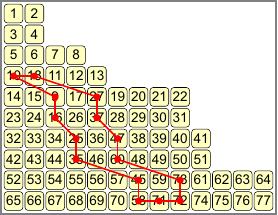}\qquad
  \includegraphics [scale=1] {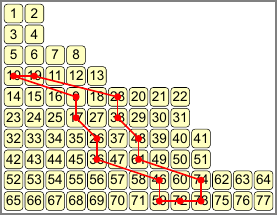}\qquad
  \includegraphics [scale=1] {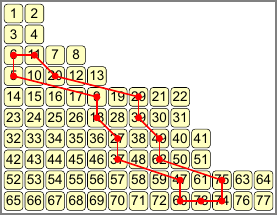}\\
  \vspace{0.5cm}
  \includegraphics [scale=1] {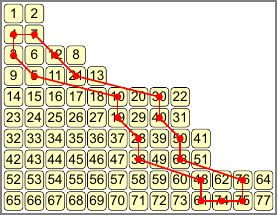}\qquad
  \includegraphics [scale=1] {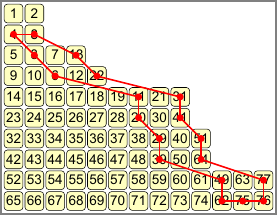}
  \caption {The type-C cycles span, among others, all the tokens that are not in the $a'$ full columns.}
  \label {fig:cycles-3}
\end {figure}

\newpage

\begin {figure}[h!]
  \centering
  \includegraphics [scale=1.25] {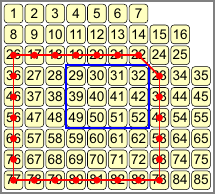}\qquad
  \includegraphics [scale=1.25] {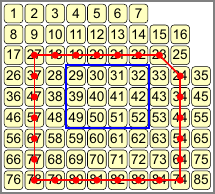}\qquad
  \includegraphics [scale=1.25] {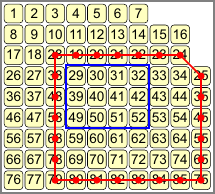}\\
  \vspace{0.5cm}
  \includegraphics [scale=1.25] {pics/b02.pdf}\qquad
  \includegraphics [scale=1.25] {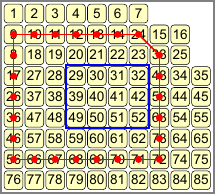}\qquad
  \includegraphics [scale=1.25] {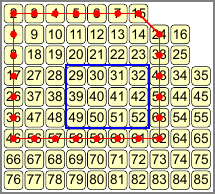}\\
  \vspace{0.5cm}
  \includegraphics [scale=1.25] {pics/b02.pdf}\qquad
  \includegraphics [scale=1.25] {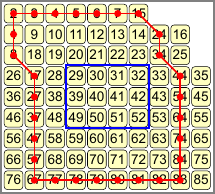}\qquad
  \includegraphics [scale=1.25] {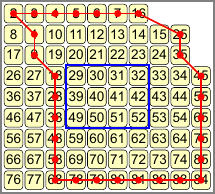}
  \caption {The type-A cycles (top row), the type-B cycles (middle row), and the type-C cycles (bottom row) collectively span all the non-core tokens.}
  \label {fig:cycles-4}
\end {figure}

\subsubsection*{Special Configurations}
We will now discuss all the configurations of the Permutation Puzzle not covered by Theorem~\ref{theorem-gen1}.

\begin{itemize}
\item If $a''=b''=0$, i.e., there are no empty cells in the bounding box, then clearly no token can be moved, and $G_C$ is the trivial permutation group (Figure~\ref{fig:reconf-1}, left).
\item Let $a''=b''=1$, i.e., there is exactly one empty cell, located at the top-right corner of the bounding box. The core includes all the tokens, except the ones on the perimeter of the bounding box, which are spanned by the type-A cycle with $k=0$. It is easy to see that the only possible permutations are iterations of this cycle and its inverse. Therefore, $G_C$ is isomorphic to the cyclic group $C_{2a+2b-5}$ (Figure~\ref{fig:reconf-1}, right).
\end{itemize}

\begin {figure}
\centering
  \includegraphics [scale=1.5] {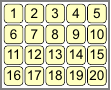}\qquad
  \includegraphics [scale=1.5] {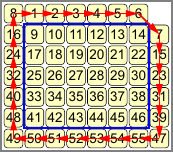}
  \caption {Two configurations with $a''b''\leq 1$.}
  \label {fig:reconf-1}
\end {figure}

In the following, we will assume that $a''\geq b''$ and $a''\geq 2$, and we will discuss all configurations where $n/2< a'+b'+1$. We will invoke some theorems from~\cite{SVU22} about \emph{cyclic shift puzzles}.

\begin{itemize}
\item Let $b=2$, and let the two rows have length $1$ and $3$, respectively (Figure~\ref{fig:reconf-2}, top row, first image). The two type-A cycles $(1\ 2\ 3)$ and $(1\ 3\ 4)$ form a 2-connected $(3,3)$-puzzle involving all tokens, and therefore $G_C=\Alt{n}$, due to~\cite[Theorem~2]{SVU22}.
\item Let $b=2$, and let the two rows have length $1$ and $4$, respectively (Figure~\ref{fig:reconf-2}, top row, second image). The first and third type-A cycles $(1\ 2\ 3)$ and $(1\ 4\ 5)$ form a 1-connected $(3,3)$-puzzle involving all tokens, and therefore $G_C=\Alt{n}$, due to~\cite[Theorem~1]{SVU22}.
\item Let $b=2$, and let the two rows have length $2$ and $4$, respectively (Figure~\ref{fig:reconf-2}, top row, third image). We denote the two type-A cycles by $\alpha=(1\ 3\ 4\ 5\ 2)$ and $\beta=(1\ 4\ 5\ 6\ 2)$. It is easy to see that $G_C$ is generated by $\alpha$ and $\beta$, because any non-trivial sequence of four pushes necessarily goes through a configuration whose canonical form yields one the permutations $\alpha$, $\beta$, $\alpha^{-1}$, or $\beta^{-1}$.

In order to determine $G_C$, we transform $\alpha$ and $\beta$ by a suitable outer automorphism $\psi\colon \Sym{6}\to \Sym{6}$. Since $\psi$ is an automorphism, the group $G'_C$ generated by $\psi(\alpha)$ and $\psi(\beta)$ is isomorphic to $G_C$. The automorphism $\psi$ is defined on a set of generators of $\Sym{6}$ as follows (cf.~\cite[Corollary~7.13]{rotman}):
\begin{align*}
\psi((1\ 2))&=(1\ 5)(2\ 3)(4\ 6),\\
\psi((1\ 3))&=(1\ 4)(2\ 6)(3\ 5),\\
\psi((1\ 4))&=(1\ 3)(2\ 4)(5\ 6),\\
\psi((1\ 5))&=(1\ 2)(3\ 6)(4\ 5),\\
\psi((1\ 6))&=(1\ 6)(2\ 5)(3\ 4).
\end{align*}
We have
\begin{align*}
\alpha&=(1\ 3\ 4\ 5\ 2) = (1\ 2) (1\ 5) (1\ 4) (1\ 3),\\
\beta&=(1\ 4\ 5\ 6\ 2) = (1\ 2) (1\ 6) (1\ 5) (1\ 4).
\end{align*}
Therefore, the two generators of $G'_C$ are
\begin{align*}
\psi(\alpha)&=\psi((1\ 2)) \psi((1\ 5)) \psi((1\ 4)) \psi((1\ 3)) = (1\ 6\ 2\ 3\ 5),\\
\psi(\beta)&=\psi((1\ 2)) \psi((1\ 6)) \psi((1\ 5)) \psi((1\ 4)) = (1\ 3\ 2\ 6\ 5).
\end{align*}
Both generators $\psi(\alpha)$ and $\psi(\beta)$ leave the number $4$ fixed, and thus $G'_C$ is isomorphic to a subgroup of $\Sym{5}$. Moreover, the generators are cycles of odd length, and therefore they produce only even permutations. Hence, $G'_C$ is isomorphic to a subgroup of $\Alt{5}$. Observe that $\psi(\alpha)\psi(\beta)=(1\ 5\ 6)$, which is a 3-cycle involving consecutive elements of the 5-cycle $\psi(\alpha)$. These two cycles generate all even permutations on $\{1,2,3,5,6\}$ (cf.~\cite[Proposition~1]{SVU22}).

\begin {figure}
\centering
  \includegraphics [scale=1.5] {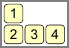}\qquad
  \includegraphics [scale=1.5] {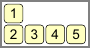}\qquad
  \includegraphics [scale=1.5] {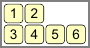}\qquad
  \includegraphics [scale=1.5] {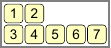}\\
  \vspace{0.5cm}
  \includegraphics [scale=1.5] {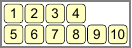}\qquad
  \includegraphics [scale=1.5] {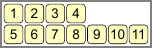}\qquad
  \includegraphics [scale=1.5] {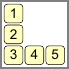}\qquad
  \includegraphics [scale=1.5] {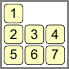}
  \caption {Some special configurations with $b=2$ and $b=3$ rows.}
  \label {fig:reconf-2}
\end {figure}

We conclude that $G'_C$, and therefore $G_C$, is isomorphic to $\Alt{5}$, which is a group of order $60$ and index $12$ in $\Sym{6}$. A permutation $\pi\in\Sym{6}$ is in $G_C$ if and only if $\psi(\pi)$ is even and leaves the number $4$ fixed.
\item Let $b=2$, and let the two rows have length $2$ and $5$, respectively (Figure~\ref{fig:reconf-2}, top row, fourth image). We denote the three type-A cycles by $\alpha=(1\ 3\ 4\ 5\ 2)$, $\beta=(1\ 4\ 5\ 6\ 2)$, and $\gamma=(1\ 5\ 6\ 7\ 2)$. Consider the permutations $(\alpha\beta^{-1}\gamma)^2=(2\ 5\ 4)$ and $\beta\gamma\alpha=(1\ 3\ 5\ 4\ 2\ 6\ 7)$. We have a 3-cycle involving consecutive elements of a 7-cycle, which generate all even permutations on $\{1,2,\dots,7\}$ (cf.~\cite[Proposition~1]{SVU22}). We conclude that $G_C=\Alt{n}$.
\item Let $b=2$, and let the two rows have length $\ell\geq 3$ and $\ell+2$, respectively (Figure~\ref{fig:reconf-2}, bottom row, first image). We denote the two type-A cycles by $\alpha=(1, \ell+1, \dots, n-1, \ell, \dots, 2)$ and $\beta=(1, \ell+2, \dots, n, \ell, \dots, 2)$. The permutation $\beta^{-2}(\alpha^2\beta^{-1}\alpha^{-2}\beta)^2\beta^2=(n-3, n-2, n)$ is a 3-cycle that, together with $\alpha$, forms a 2-connected $(3,n-1)$-puzzle involving all tokens. We conclude that $G_C=\Alt{n}$, due to~\cite[Theorem~2]{SVU22}.
\item Let $b=2$, and let the two rows have length $\ell\geq 3$ and $\ell+3$, respectively (Figure~\ref{fig:reconf-2}, bottom row, second image). Observe that this configuration is the same as the previous one, except for an extra token, labeled $n$, in the bottom row. In particular, the first two type-A cycles are the same, and generate all even permutations on $\{1,2,\dots n-1\}$, including the 3-cycle $(1,\ell+1,\ell+2)$. This 3-cycle forms a 1-connected $(3,n-2)$-puzzle with the third type-A cycle. Since all tokens are involved in this puzzle, we conclude that $G_C=\Alt{n}$, due to~\cite[Theorem~1]{SVU22}.
\item Let $b=3$, and let the three rows have length $1$, $1$, and $3$, respectively (Figure~\ref{fig:reconf-2}, bottom row, third image). The two type-A cycles $(2\ 3\ 4)$ and $(2\ 4\ 5)$ form a 2-connected $(3,3)$-puzzle, and therefore they generate all even permutations on $\{2,3,4,5\}$, due to~\cite[Theorem~2]{SVU22}. In particular, they generate the 3-cycle $(3\ 4\ 5)$, which forms a 1-connected $(3,3)$-puzzle with the type-B cycle $(1\ 2\ 4)$, involving all tokens. By~\cite[Theorem~1]{SVU22}, $G_C=\Alt{n}$.
\item Let $b=3$, and let the three rows have length $1$, $3$, and $3$, respectively (Figure~\ref{fig:reconf-2}, bottom row, fourth image). We denote the two type-A cycles by $\alpha=(1\ 2\ 5\ 6\ 3)$ and $\beta=(1\ 3\ 6\ 7\ 4)$. Consider the permutations $(\beta^2\alpha^{-1})^2=(2\ 5\ 6)$ and $\alpha\beta^{-1}=(1\ 4\ 7\ 3\ 2\ 5\ 6)$. We have a 3-cycle involving consecutive elements of a 7-cycle, which generate all even permutations on $\{1,2,\dots,7\}$ (cf.~\cite[Proposition~1]{SVU22}). We conclude that $G_C=\Alt{n}$.
\end{itemize}

\begin {figure}
\centering
  \includegraphics [scale=1.5] {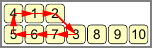}\qquad
  \includegraphics [scale=1.5] {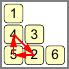}\qquad
  \includegraphics [scale=1.5] {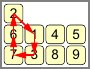}\\
  \vspace{0.5cm}
  \includegraphics [scale=1.5] {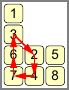}\qquad
  \includegraphics [scale=1.5] {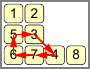}\qquad
  \includegraphics [scale=1.5] {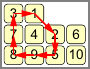}
  \caption {Some small configurations where at least three tokens are not covered by the first type-A cycle.}
  \label {fig:reconf-3}
\end {figure}

We will now prove that all configurations not listed above satisfy $n/2\geq a'+b'+1$, i.e., they have at least three tokens not lying on the first type-A cycle.
\begin{itemize}
\item If $b=1$, then $a''=b''=0$. This case has already been discussed (Figure~\ref{fig:reconf-1}, left).
\item For $b=2$, we have discussed all cases where $a''\leq 3$. If $a''\geq 4$, then $n=2a'+a''$, while a type-A cycle spans $2a'+1$ tokens, and leaves out at least three tokens (Figure~\ref{fig:reconf-3}, top row, first image).
\item Let $b\geq 3$ and $a'=b'=1$. Then, a type-A cycle spans three tokens. We are assuming that $a''\geq 2$, and so $a\geq 3$, implying that $n\geq 5$. The only case where $n=5$ has already been discussed (it is the case with $b=3$ rows of length $1$, $1$, and $3$, illustrated in Figure~\ref{fig:reconf-2}, bottom row, third image); in all other cases we have $n\geq 6$, and thus at least three tokens are left out of the type-A cycle (Figure~\ref{fig:reconf-3}, top row, second image).
\item Let $b\geq 3$, $a'=1$, and $b'=2$. Then, a type-A cycle spans five tokens. We are assuming that $a''\geq 2$, and so $a\geq 3$, implying that $n\geq 7$. The only case where $n=7$ has already been discussed (it is the case with $b=3$ rows of length $1$, $3$, and $3$, illustrated in Figure~\ref{fig:reconf-2}, bottom row, fourth image); in all other cases we have $n\geq 8$, and thus at least three tokens are left out of the type-A cycle (Figure~\ref{fig:reconf-3}, top row, third image).
\item Let $b\geq 3$, $a'=1$, and $b'\geq 3$ (Figure~\ref{fig:reconf-3}, bottom row, first image). Since we are assuming that $a''\geq 2$, the rightmost column contains at least three tokens, which are not included in the first type-A cycle.
\item Let $b\geq 3$, $a'\geq 2$, and $b'=1$ (Figure~\ref{fig:reconf-3}, bottom row, second image). Since we are assuming that $a''\geq 2$, the rightmost token is not included in the first type-A cycle. Moreover, the top row contains at least two tokens, which are not in the type-A cycle. In total, there are at least three tokens not spanned by the cycle.
\item Let $b\geq 3$, $a'\geq 2$, and $b'\geq 2$ (Figure~\ref{fig:reconf-3}, bottom row, third image). Since we are assuming that $a''\geq 2$, the rightmost column contains at least two tokens, which are not included in the first type-A cycle. Moreover, the token at $(1,1)$ is not contained in the first type-A cycle, either. In total, there are at least three tokens not spanned by the cycle.
\end{itemize}

\subsubsection*{Arbitrary Labels}

We now discuss the case where not all labels are distinct. Clearly, if all the non-core tokens have distinct labels, then our previous analysis carries over verbatim.

Let us now assume that at least two non-core tokens $x$ and $y$ have the same label. In this case, we identify two permutations $\pi_1,\pi_2\in G_C$ if the configurations $C_1,C_2\colon \mathcal L\to \Sigma\cup\{{\rm empty}\}$ they produce are equal, i.e., $C_1(p)=C_2(p)$ for all $p\in\mathcal L$.

Assume that $G_C$ contains all even permutations of the non-core tokens, which we know to be always the case except in some special configurations. Let $\pi$ be any permutation of the non-core tokens. If $\pi$ is even, then $\pi\in G_C$, and we can reconfigure the tokens to match $\pi$. If $\pi$ is odd, then let $\pi'=(x\ y)\pi$. Since $\pi'$ is even, we have $\pi'\in G_C$. However, $\pi$ and $\pi'$ produce equal configurations, because $x$ and $y$ have the same label, and therefore we can reconfigure the tokens to match $\pi$, as well.

We now have a complete solution to the Permutation Puzzle.

\begin{theorem}\label{theorem-final}
If the configuration $C$ is compact, then the group $G_C$ of possible permutations is as follows.
\begin{itemize}
\item If no cell in the bounding box is empty, then $G_C$ is the trivial group.
\item If exactly one cell in the bounding box is empty, then $G_C$ is generated by the cycle of the non-core tokens taken in clockwise order.
\item If there are exactly $6$ tokens and exactly $2$ empty cells in the bounding box, then $G_C$ is isomorphic to the alternating group $\Alt{5}$.
\item In all other cases, $G_C$ is the alternating group on the non-core tokens. Hence, if at least two non-core tokens have the same label, then all permutations of the non-core labels can be obtained; otherwise, only the even permutations of the non-core labels can be obtained.\qed
\end{itemize}
\end{theorem}

\section{Conclusions and Open Problems}\label{s:5}
Concerning the Compaction Puzzle, we showed that all sparse configurations of $n=ab$ tokens can be pushed into an $a\times b$ box if and only if $a\leq 2$ or $b\leq 2$ or $a=b=3$ (Theorem~\ref{thm:compaction}), and that there exist sparse configurations capable of becoming any compact configuration of size $n$ (Observation~\ref{obs:universal-conf}). 
We leave the following as an open problem.

\begin{openproblem}
	\label{op:shape}
Is it NP-complete to decide whether a given unlabeled sparse configuration can be pushed into a given compact configuration?
\end{openproblem}

We do not know the answer to this problem even if we restrict the target configuration to be a rectangle (note that the NP-completeness proof in~\cite{trash} does not hold for sparse initial configurations).

Concerning the Permutation Puzzle, we have shown that, given two same-shaped compact configurations, all even permutation of the non-core tokens are reachable, with a few minor exceptions (Theorem~\ref{theorem-final}). Notably, if at least two tokens have the same label, then \emph{all} permutations of the non-core labels can be obtained.

In particular, the puzzle in Figure~\ref{fig:fun2022-puzzle} is solvable, because it features a compact configuration with no core tokens, more than two empty cells, and at least two same-labeled (i.e., same-colored) tokens. Thus, according to Theorem~\ref{theorem-final}, any permutation of the labels is feasible in this puzzle, and we only have to verify that the number of tokens with any given label in the initial configuration matches the number of tokens with that label in the goal configuration.

We remark that, even though our proof relies on Theorem~\ref{theorem-jones}, which is a deep, non-constructive result, we have nonetheless uncovered a great deal of structure in the Permutation Puzzles. We claim that, when a Permutation Puzzle is solvable, there is a solution within $O(n^4)$ pushes which is computable by a polynomial-time algorithm. However, we conjecture that finding the \emph{shortest} solution is NP-hard (this is known to be the case in other token-shifting puzzles~\cite{SVU22,torus,15puzzle}).

\begin{openproblem}
Is it NP-hard to find the shortest solution in the Permutation Puzzle?
\end{openproblem}

When the initial configuration is not compact, the pushes applied in order to obtain a compact one can affect the permutation, making the problem substantially harder. 
For example, the pushes performed before the configuration becomes incompressible may affect the configuration of core tokens.
Furthermore, a solution to the general problem of transforming an arbitrary labeled configuration into a target compact one would imply a solution to Open Problem~\ref{op:shape}.

\begin{openproblem}
	Given a (not necessarily compact) labeled configuration, is there a sequence of pushes that transforms it into a (not necessarily compact) target configuration?
\end{openproblem}

\paragraph{Acknowledgments.} The authors would like to thank the anonymous reviewers for helpful suggestions that greatly improved the readability of this paper.

{\small
\bibliographystyle{abbrv}
\bibliography{references}
}

\end{document}